\documentclass[12pt,reqno]{amsart}
\usepackage{graphicx} 
\usepackage{amsmath, amsfonts, amssymb, amsthm}
\usepackage{amsmath,amssymb,amscd,verbatim,color}
\usepackage[backref]{hyperref}
\usepackage{MnSymbol}
\usepackage{stmaryrd}
\usepackage{amsfonts}
\usepackage[top=35mm, bottom=35mm, left=30mm, right=30mm]{geometry}
\usepackage{xypic}
\usepackage{xcolor}
\usepackage{amsthm}
\usepackage{tikz}
\usepackage{tikz-cd}
\theoremstyle{plain}
\newtheorem{thm}{Theorem}

\newtheorem{prop}[thm]{Proposition}

\theoremstyle{definition}

\numberwithin{equation}{section}

  \def \g{\lambda}

  \def \G{\Lambda}

\def\R{\mathbb{R}}

\def\I{\mathcal{I}}
\def\m{\mathcal{M}}

\def\mmax{\mathcal M_{\rm max}}
\def\lip{{\rm Lip}}

\allowdisplaybreaks

\renewcommand*{\backref}[1]{}

\makeatletter

\def\MRbibitem{\@ifnextchar[\my@lbibitem\my@bibitem}

\def\mybiblabel#1#2{\@biblabel{{\hyperref{http://www.ams.org/mathscinet-getitem?mr=#1}{}{}{#2}}}}

\def\my@lbibitem[#1]#2#3#4\par{%
	\item[\mybiblabel{#2}{#1}\myhyperanchor{#3}\hfill]#4%
	\@ifundefined{ifbackrefparscan}{}{\BR@backref{#3}}%
	\if@filesw{\let\protect\noexpand\immediate
		\write\@auxout{\string\bibcite{#3}{#1}}}\fi\ignorespaces%
}

\def\my@bibitem#1#2#3\par{%
	\refstepcounter\@listctr
	\item[\mybiblabel{#1}{\the\value\@listctr}\myhyperanchor{#2}\hfill]#3%
	\@ifundefined{ifbackrefparscan}{}{\BR@backref{#2}}%
	\if@filesw\immediate\write\@auxout
	{\string\bibcite{#2}{\the\value\@listctr}}\fi\ignorespaces%
}

\makeatother

\begin{document}
\title[Typical Uniqueness in Ergodic Optimization]{Typical Uniqueness in Ergodic Optimization}
\author{Oliver~Jenkinson, Xiaoran~Li, Yuexin~Liao, and Yiwei~Zhang}

\address[O.~Jenkinson]{School of Mathematical Sciences, Queen Mary University of London, Mile End Road,
London, E1 4NS, UK}
\email{o.jenkinson@qmul.ac.uk}

\address[X.~Li]{The Division of Physics, Mathematics and Astronomy, California Institute of Technology, 1200 E California Blvd, Pasadena CA 91125, USA}
\email{xiaoran@caltech.edu}

\address[Y.~Liao]{The Division of Physics, Mathematics and Astronomy, California Institute of Technology, 1200 E California Blvd, Pasadena CA 91125, USA}
\email{yliao@caltech.edu}

\address[Y.~Zhang]{School of Mathematical Sciences and big data, Anhui University of Science and Technology, Huainan, Anhui 232001, P.R. China}
\email{zhangyw@aust.edu.cn, yiweizhang831129@gmail.com}



\thanks{We are grateful to an anonymous referee for suggesting the argument included as an Appendix. Y. Zhang is partially supported by National Natural Science Foundation (Nos.~12161141002, 12271432), and USTC-AUST Math Basic Discipline Research Center.}


\maketitle
\begin{abstract}
For ergodic optimization on any topological dynamical system,
 with real-valued potential function $f$ belonging to any separable Banach space $V$ of continuous functions,
we show that the $f$-maximizing measure is typically unique, 
in the strong sense that 
a countable collection of hypersurfaces contains the exceptional set
of
those $f\in V$ with non-unique maximizing measure.
This strengthens previous results asserting that the uniqueness set is both residual and prevalent.

For various other topological vector spaces of continuous functions, we show that uniqueness of the maximizing measure is either generic, or prevalent, or corresponds to the complement of a sparse set.
\end{abstract}

\section{Introduction}\label{introsection}

Given a topological dynamical system on a compact metrizable space $X$,
and a continuous function $f:X\to\R$,
an invariant probability measure $\mu$ is said to be \emph{$f$-maximizing}
if no other invariant probability measure gives $f$ a larger integral than $\int f\, d\mu$, the so-called \emph{maximum ergodic average}.
Early investigations of maximizing measures (see e.g.~\cite{Bou00,HO96a,HO96b,Jen00}) suggested that for 
chaotic dynamical systems, the maximizing measure for typical functions $f$ is unique and periodic. 
The problem of proving that the maximizing measure is typically \emph{periodic} has stimulated considerable work
(see e.g.~\cite{bochi,BZ16,B01,B08,Co16,CLT01,DLZ24,GSZ25,hjxz,HLMXZ,LZ24,morrisnonlinearity,quassiefken}), and remains an area of active interest;
the archetypal result is that if the dynamical system is suitably hyperbolic, 
then given a space $V$ of suitably regular (e.g.~smooth, or Lipschitz) functions, there is a subset $V_P\subseteq V$
that is large in either a topological or a probabilistic sense,
and has the property that if $f\in V_P$ then the $f$-maximizing measure is unique and periodic.

By contrast the phenomenon of typically \emph{unique} maximizing measure is more universal, and known to hold for 
rather general classes of dynamical system and function space.
For example if $T:X\to X$ is continuous, and $V$ is any topological vector space that is 
densely and continuously included, via a linear injection, in the space $C(X)$ of continuous functions on $X$,
 then a generic function in $V$ has a unique maximizing measure;
that is, if $V^!$ denotes those functions in $V$ with a unique maximizing measure, 
then $V^!$ is the intersection
of countably many open dense subsets of $V$ (see \cite{jenkinsonergodicoptimization}).
Recently Morris \cite{morris}, answering a question raised in \cite{BZ16}, has shown that uniqueness of the maximizing measure is typical in a probabilistic sense,
proving that for $T:X\to X$ as above, and $V$ a separable Fr\'echet space that
 is densely and continuously included in $C(X)$ via a linear injection, the uniqueness set $V^!$ is \emph{prevalent}
(see \cite{huntsaueryorke, huntsaueryorkeaddendum} for the definition of prevalence).

The purpose of this note is as follows: firstly to give alternative proofs,
somewhat shorter than in \cite{jenkinsonergodicoptimization, morris},
 of both the generic uniqueness and prevalent uniqueness results mentioned above,
secondly to prove a stronger result in the case that $V$ is a separable Banach space that is 
densely and continuously included in $C(X)$ via a linear injection,
namely that the set $V\setminus V^!$ of functions with non-unique maximizing measure can be covered by a countable collection of (delta-convex) hypersurfaces in $V$, and lastly (in the Appendix) to show that if $V$ is the (non-separable)  
space $\lip(X)$, then  $V\setminus V^!$  can be covered by a countable collection of Lipschitz hypersurfaces in $V$.
Throughout we shall work with very general notions of dynamical system, allowing $T$ to be multi-valued, and allowing 
group and semi-group actions (see Section \ref{typicalsection} for further details).

\subsection*{Statement of provenance}
This is an author-created, un-copyedited version of an article published in \textit{Nonlinearity}. IOP Publishing Ltd is not responsible for any errors or omissions in this version of the manuscript or any version derived from it. The Version of Record is available online at
 \url{http://dx.doi.org/10.1088/1361-6544/ae6e13}.

\section{Typical uniqueness of maximizing measures}\label{typicalsection}

For a compact metrizable space $X$,
we will consider general topological dynamical systems on $X$.
This includes in particular the continuous self-maps $T:X\to X$ mentioned in Section \ref{introsection},
but also generalisations in two directions.
The first is to let
$\G$ be a topological group or semi-group, and $\phi$ a topological
action of $\G$ on $X$ (i.e.~a continuous map $\phi:\G\times X\to X$,
$(\g,x)\mapsto \phi_\g(x)$ such that $\phi_1=\text{id}_X$ and
$\phi_{\g'}\circ\phi_\g=\phi_{\g' \g}$ for all $\g,\g'\in \G$); 
in this setting a Borel
probability measure $\mu$ on $X$ is said to be invariant if
$\mu(\phi_\g^{-1}A)=\mu(A)$ for all $\g\in \G$ and all Borel sets $A$. 
The second generalisation is to let $T$ be a set-valued mapping $T:X\to 2^X$
that is upper semi-continuous (see e.g.~\cite{akin}); here  a Borel
probability measure $\mu$ on $X$ is said to be invariant
if $\mu(A)\le \mu(T^{-1}A)$ for all Borel sets $A$
(see e.g.~\cite{millerakin}).

 Let
$\m$ denote the set of invariant Borel probability measures.
We shall always assume that $\m$ is non-empty: 
this is well known to be the case if the dynamical system is generated by a single continuous map $T$
(or is a flow or semi-flow), by the Krylov-Bogolioubov theorem \cite{kb} (see \cite[Cor.~6.9.1]{walters}), 
or more generally  whenever $\G$ is amenable (see
e.g.~\cite[p.~97]{glasner}).  
In the case of multi-valued $T$, the set $\m$ is non-empty if and only if 
$\{ (x_i)_{i \in \mathbb{Z}} \in X^{\mathbb{Z}} : x_{i+1} \in T (x_i) \text{ for all } i \in \mathbb{Z} \}$
is non-empty
(see e.g.~\cite{jllz,millerakin}).

When equipped with the weak$^*$ topology, $\m$ is compact and metrizable.
Let $C(X)$ denote the set of continuous real-valued functions defined on $X$, equipped with the supremum norm.
We say that $\mu\in\m$ is
\emph{$f$-maximizing} for $f\in C(X)$ if
$$
\beta(f):= \sup \left\{ \int f\, \mathrm{d}\nu : \nu\in \m\right\} = \int f\, \mathrm{d}\mu.
$$
The set $\m_{\max} (f)$ of $f$-maximizing measures is convex, closed in $\m$, and each 
of its extreme points is also extreme\footnote{If the dynamical system is a group or semi-group action,
then $\m$ is a simplex, and the extreme points of $\m$ are precisely the ergodic invariant measures; if the dynamical system is given by a multi-valued mapping then $\m$ need not be a simplex (see e.g.~\cite{jllz}).} in $\m$.
See e.g.~\cite{Boc18,new survey} for general background on maximizing measures, and the wider area of ergodic optimization.

Throughout, when we say that a topological vector space $V$
is \emph{continuously included} in $C(X)$, we mean
that there exists a continuous injective linear map $\iota:V\to C(X)$;
 such an $\iota$ is referred to as an 
\emph{inclusion}\footnote{Alternative terminology, common in functional analysis, is that $\iota$
is an \emph{embedding}, though we do not use this term here, 
to avoid confusion with its alternative definition, common in topology, to mean a homeomorphism onto its image (see e.g.~\cite{aliprantisborder}).}
map, and we identify $V$ with its image $\iota(V)\subseteq C(X)$,
so that elements of $V$ are regarded as continuous functions on $X$.
Such a $V$ is said to be
 \emph{densely and continuously included} in $C(X)$ if there exists an 
inclusion $\iota:V\to C(X)$
whose image is dense in $C(X)$.

We first give an alternative, shorter, proof of the following result from \cite[Thm.~3.2]{jenkinsonergodicoptimization}.

\begin{thm}\label{tvs_residual}
If $V$ is a topological vector space that is densely and continuously included in $C(X)$,
then $$V^{!}= \{f\in V: f\text{ has a unique maximizing measure} \}$$
is 
a residual subset of $V$.
\end{thm}
\begin{proof}
Defining $\gamma :V\to 2^{\m}$ by
 $\gamma (f) =\m_{\max} (f)$, we first claim that $\gamma$
is upper semi-continuous.
    Now $\m$ is compact and metrizable,
 and $\gamma (f)$ is closed in $\m$ for all $f \in V$. 
The upper semi-continuity of $\gamma$ is equivalent to the closedness of the graph
$
        gr (\gamma) := \{ (f,\mu) : f \in V ,\, \mu \in \gamma (f) \}
$
    in $V \times \m$ (see \cite[Thm~16.12]{aliprantisborder}).
For any net $\{(f_i,\mu_i)\}_{i\in\I}$ in $gr(\gamma)$
converging to $(f_0,\mu_0)\in V\times\m$,
the fact that $V$ is continuously included in $C(X)$
means that $f_i\to f_0$ in $C(X)$.
For each $i$, the fact that $(f_i,\mu_i)\in gr(\gamma)$ means that
$\mu_i\in\m_{\max}(f_i)$, so
\[
  \int\! f_i\,\mathrm{d}\mu
  \;\leqslant\;
  \int\! f_i\,\mathrm{d}\mu_i
\]
for all $\mu\in\mathcal{M}$.
Taking the limit over the net gives
\[
  \int\! f_0\,\mathrm{d}\mu \;\leqslant\; \int\! f_0\,\mathrm{d}\mu_0
\]
for all $\mu\in\m$, so $\mu_0\in\m_{\max}(f_0)=\gamma(f_0)$
and $(f_0,\mu_0)\in gr(\gamma)$.
Thus $gr(\gamma)$ is closed in $V\times\m$,
so $\gamma$ is indeed upper semi-continuous.

A theorem of Fort \cite{fort2} (cf.~\cite[p.~61]{engelking}, \cite{Mo13}) asserts that an upper semi-continuous
map from a topological space to the set of non-empty compact subsets of a metric space
has the property that its set of continuity points is residual.
Our theorem will follow, therefore, if it can be shown that
the set of continuity points for $\gamma:V\to 2^{\m}$ is precisely $V^!$.

Recalling (see \cite[Defn.~16.2]{aliprantisborder}) that lower semi-continuity of $\gamma$  at $f$ means that for every open subset $U$ of $\m$ intersecting $\gamma (f)$, 
the set $\gamma^{-1} (U) := \{ h \in V : \gamma (h) \cap U \neq \emptyset \}$ is a neighbourhood of $f$ in $V$,
we 
claim that $\gamma$ is lower semi-continuous at $f$ if and only if $f\in V^!$.
For this, assume first that $f \in V^!$.
If $U$ is an open set in $\m$ intersecting with $\gamma (f)$, then $U$ contains the singleton $\gamma (f)$. 
Upper semi-continuity of $\gamma$ then
implies that there is an open neighbourhood $U_f$ of $f$ in $V$ such that $\gamma (g) \subseteq U$ for every $g \in U_f$. 
Now $\gamma (g) =\m_{\max} (g)$ is non-empty for all $g \in V$, so $\gamma (g) \cap U \neq \emptyset$ for all $g \in U_f$. So $U_f \subseteq \gamma^{-1} (U)$, and therefore $\gamma$ is lower semi-continuous at $f$.

    If $f \in V\setminus V^!$ then choose distinct measures $\mu_1,\, \mu_2 \in \m_{\max} (f) =\gamma (f)$,
and some $g \in V$ with $\int \! g \, \mathrm{d} \mu_1 > \int \! g \, \mathrm{d} \mu_2$, the existence 
of $g$ being ensured by the denseness of $V$ in $C(X)$. 
If
    $
        U := \{ \mu \in \m : \int \! g \, \mathrm{d} \mu < \int \! g \, \mathrm{d} \mu_1 \}
    $,
    then $U$ is an open subset of $\m$ intersecting with $\gamma (f)$. 
If $\mu \in U$ then $\int \! f \, \mathrm{d} \mu_1 \geqslant \int \! f \, \mathrm{d} \mu$ and $\int \! g \, \mathrm{d} \mu_1 > \int \! g \, \mathrm{d} \mu$, so
$
        \int \! (f+ \epsilon g) \, \mathrm{d} \mu_1 > \int \! (f+ \epsilon g) \, \mathrm{d} \mu
    $
 for all $\epsilon >0$.
    But $\mu \in U$ was arbitrary, so $U \cap \gamma (f +\epsilon g) = U \cap \m_{\max} (f+ \epsilon g) =\emptyset$ for every $\epsilon >0$, thus $\gamma^{-1} (U)$ is not a neighbourhood of $f$ in $V$,
and hence $\gamma$ is not lower semi-continuous at $f$.
\end{proof}

When $V$ is a separable Banach space, Theorem \ref{tvs_residual} can be strengthened as
in the following Theorem \ref{hypersurfaces}.
To state it, a subset $S$ of $V$ 
will be called a \emph{hypersurface}
(see \cite[p.~92]{BL00})
if it 
is the graph of a function, defined on a hyperplane (i.e.~codimension-one subspace) in $V$, that is the difference of two Lipschitz convex functions.
That is, $S$ can be written as
$S=\{f + (\phi(f)-\psi(f))g:f\in W\}$,
where $g\in V$ and  $W$ is a 
closed subspace of $V$
with $V=W+\{tg:t\in\R\}$,
and $\phi,\psi$ are Lipschitz continuous convex functions on $V$.
Such an $S$ is sometimes referred to as a \emph{delta-convex hypersurface}
(see e.g.~\cite{pavlica, veselyzajicek, zajicek3}),
or a
\emph{(c-c) hypersurface}
(see \cite[Defn.~2]{zajicek1}).

\begin{thm}\label{hypersurfaces}
If $V$ is a separable Banach space that is densely and continuously included in $C(X)$,
then 
$$V^{!}= \{f\in V: f\text{ has a unique maximizing measure} \}$$
is 
such that its complement 
can be covered by countably many 
hypersurfaces.
\end{thm}
\begin{proof}
If $\Vert \cdot \Vert_{\infty}$ denotes the norm in $C(X)$,
then
$\bigl| \int \! f \, \mathrm{d} \mu -\int \! g \, \mathrm{d} \mu \bigr| \leqslant \Vert f-g \Vert_{\infty}$ for all $f ,\, g \in V$,
$\mu \in \m$, so $|\beta (f) -\beta (g) | \leqslant \Vert f-g \Vert_{\infty}$. 
But $V$ is continuously included in $C(X)$, 
so $\beta$ is Lipschitz continuous. 
Now $\beta$ is also convex, since if $f,\, g \in V$, $a \in [0,1]$, and $\mu \in \mmax (af +(1-a)g)$, then
$
        \beta (af +(1-a)g) =a \int \! f \, \mathrm{d} \mu +(1-a) \int \! g \, \mathrm{d} \mu \leqslant a \beta (f) +(1-a) \beta (g)
$.

By a theorem of Zaj\' i\v cek \cite{zajicek1} (see also e.g.~\cite[Thm.~4.20]{BL00}), a continuous convex function on 
a separable Banach space
is Gateaux differentiable except on a
set that  can be covered by
countably
 many  hypersurfaces.
It suffices to show, therefore, that
$f\in V^!$ if and only if $f$ is a point of Gateaux differentiability for
$\beta$.

    For this, let $f,\, g \in V$, and $\mu \in \mmax (f)$, and note that
$
        \beta (f+g) \geqslant \int \! (f+g) \, \mathrm{d} \mu =\beta (f) +\int \! g \, \mathrm{d} \mu
$.
    Let $\tau >0$. Replacing $(f,g)$ with $(f, \tau g)$ and $(f+ \tau g, -\tau g)$ gives
    \begin{equation}\label{g3r89bu}
        \sup_{\mu \in \m_{\max}(f)} \int \! g \, \mathrm{d} \mu \leqslant \frac{\beta(f + \tau g) - \beta(f)}{\tau} \leqslant \sup_{\mu \in \m_{\max} (f+\tau g)} \int \! g \, \mathrm{d} \mu.
    \end{equation}

First suppose that $f\in V^!$, with  $\m_{\max} (f) = \{ \mu_0 \}$, say.
As noted in the proof of Theorem \ref{tvs_residual},
the map that
sends each element of $V$ to its set of maximizing measures
is continuous at $f$. Letting $\tau \to 0^+$ in \eqref{g3r89bu} implies
    \begin{equation*}
        \lim_{\tau \to 0^+} \frac{\beta(f + \tau g) - \beta(f)}{\tau} 
        = \int g \, \mathrm{d} \mu_0.
    \end{equation*}
    By substituting $g$ with $-g$, we obtain:
    \begin{equation*}
        \lim_{\tau \to 0^-} \frac{\beta(f + \tau g) - \beta(f)}{\tau} 
        = \int g \, \mathrm{d} \mu_0.
    \end{equation*}
    The above two limits are equal, so the directional derivative of $\beta$ at $f$ in the direction $g$ exists and is equal to $\int \! g \, \mathrm{d} \mu_0$. But $g \in V$
was arbitrary, so
$\beta$ is Gateaux differentiable at $f$.

Now suppose that $f\in V\setminus V^!$, so
there exist distinct measures $\mu_1, \mu_2 \in \m_{\max}(f)$, and since $V$ is dense in $C(X)$, 
there exists $g \in V$ such that $\int \! g \, \mathrm{d} \mu_1 \neq \int \! g \, \mathrm{d} \mu_2$.
    Now \eqref{g3r89bu} implies
    \begin{equation*}
        \liminf_{\tau \to 0^+} \frac{\beta(f + \tau g) - \beta(f)}{\tau} \geqslant \sup_{\mu \in \m_{\max} (f)} \int \! g \, \mathrm{d} \mu.
    \end{equation*}
    By substituting $g$ with $-g$, we obtain
    \begin{equation*}
        \limsup_{\tau \to 0^-} \frac{\beta(f + \tau g) - \beta(f)}{\tau} \leqslant \inf_{\mu \in \m_{\max} (f)} \int \! g \, \mathrm{d} \mu.
    \end{equation*}
    Since $\int \! g \, \mathrm{d} \mu_1 \neq \int \! g \, \mathrm{d} \mu_2$, we have $\sup\limits_{\mu \in \m_{\max}(f)} \int \! g \, \mathrm{d} \mu > \inf\limits_{\mu \in \m_{\max}(f)} \int \! g \, \mathrm{d} \mu$, so
    \[
    \liminf_{\tau \to 0^+} \frac{\beta(f + \tau g) - \beta(f)}{\tau} > \limsup_{\tau \to 0^-} \frac{\beta(f + \tau g) - \beta(f)}{\tau},
    \]
    so the directional derivative of $\beta$ in the direction $g$ does not exist at $f$. That is, $\beta$ is not Gateaux differentiable at $f$.
\end{proof}

Note that for $V$ a separable Banach space, as in Theorem \ref{hypersurfaces}, Conze \& Guivarc'h \cite{conzeguivarch}
noted that the weaker conclusion of $V^!$ being residual in $V$ (i.e.~as in Theorem \ref{tvs_residual}) is a consequence
of a theorem of Mazur \cite{mazur}, together with the observation that $\m_{\max}(f)$ consists of the set of tangent
functionals to $\beta$ at $f$. 
In thermodynamic formalism, where the set of tangent functionals to the pressure functional $P$ at a given $f$ can be identified with the set of equilibrium states for $f$, analogous genericity results appear in \cite{gallavottimiraclesole, israel, israelphelps, ruelle}.

If $\m$ has only countably many extreme points, the hypothesis of Theorem \ref{hypersurfaces} can be weakened and its conclusion strengthened. In the following, by a hyperplane we mean a codimension-one subspace.

\begin{prop}\label{countable_prop}
If $\m$ has countably many extreme points, and
 $V$ is a
topological vector space that is densely and continuously
included  in
$C(X)$, then 
$$V^{!}
= \{f\in V: f\text{ has a unique maximizing measure} \}$$
 is  the complement of a set that can be covered by countably many hyperplanes.
\end{prop}
\begin{proof}
Each $\mu\in\m$ can be considered as a continuous linear functional on $C(X)$, via $\mu(f)=\int\! f d\mu$, 
and its restriction to $V$ is also continuous, since $V$ is continuously included in $C(X)$.
Let $\mathcal{E}$ be the countable set of extreme points of $\m$, and 
for $\mu,\nu\in\mathcal{E}$, with $\mu\neq\nu$, let 
$V_{\mu \nu}:=\{f\in V: \int \! f \, d\mu = \int \! f\, d\nu\}$,
the kernel of $\mu-\nu$. 
Now $\mu -\nu$ is not the zero functional on $V$, since $V$ is densely included in $C(X)$,
so $V_{\mu \nu}$ is a proper closed subspace of $V$.

If $f\in V\setminus V^!$ then $\int \! f \, d\mu = \int \! f\, d\nu=\beta(f)$ for some $\mu,\nu\in\mathcal{E}$ with $\mu\neq\nu$, therefore
$
 V\setminus V^! 
$
is contained in the countable union
(over $\mu,\nu\in\mathcal{E}$, $\mu\neq\nu$)
 of the hyperplanes
$V_{\mu \nu}$, as required.
\end{proof}

The class of sets that can be covered by countably many hyperplanes from Proposition 
\ref{countable_prop}, and the class of sets that can be covered by countably many hypersurfaces
from Theorem \ref{hypersurfaces}, are examples of classes of \emph{negligible sets} or \emph{null sets}
(see e.g.~\cite[p.~167]{BL00}, where  these classes appear as the smallest, and next smallest, in a hierarchy
of $\sigma$-ideals of negligible sets).
Another notion, due to
Christensen \cite{christensen1, christensen2}, is that of \emph{Haar null sets} (see e.g.~\cite[Ch.~6.1]{BL00} and \cite{elekesnagy} for an overview).
The class of Haar null sets is larger than the classes mentioned above
(in fact it is the largest class in the hierarchy from \cite[p.~167]{BL00};
see also e.g.~\cite[pp.~15--16]{PZ01}, \cite[p.~342]{zajicek1} for related discussion).
Haar null sets are precisely the \emph{shy} sets introduced by Hunt, Sauer \& Yorke
(see \cite{huntsaueryorke, huntsaueryorkeaddendum}), who defined a subset to be \emph{prevalent}
if it is the complement of a shy set.
Morris  \cite[Thm.~1]{morris}, answering a question raised in \cite{BZ16}, showed that
$V^{!}$ is a prevalent subset of $V$;
the following Theorem \ref{frechet_prevalent} gives
an alternative proof of that result.
Note that the hypothesis is 
stronger than in Theorem \ref{tvs_residual}, and weaker than in Theorem \ref{hypersurfaces};
the conclusion is weaker than in Theorem \ref{hypersurfaces}
(though not necessarily stronger than in Theorem \ref{tvs_residual}, since prevalent sets are not necessarily residual).

\begin{thm}\label{frechet_prevalent}
If $V$ is a separable Fr\'echet space densely and continuously included in $C(X)$,
then $$V^{!}= \{f\in V: f\text{ has a unique maximizing measure} \}$$
is 
a prevalent subset of $V$.
\end{thm}
\begin{proof}
As in the proof of Theorem \ref{hypersurfaces},
$|\beta (f) -\beta (g) | \leqslant \Vert f-g \Vert_{\infty}$ 
for all $f ,\, g \in V$.
But $V$ is continuously included in $C(X)$, 
so $\Vert \cdot \Vert_{\infty}$ is a continuous semi-norm on $V$,
therefore $\beta:V\to\R$ is Lipschitz continuous. 
By a theorem of
Christensen (see \cite[Thm.~2]{christensen2}), a Lipschitz continuous
real-valued map on a separable Fr\'echet space is Gateaux differentiable at all points except for a Haar null set, i.e.~ a shy set.
The theorem then follows from the fact that 
the points of Gateaux differentiability of
$\beta$ are precisely the elements of $V^!$, and this can be shown exactly as in the proof of 
 Theorem \ref{hypersurfaces}.
\end{proof}

\section*{Appendix: Extension to non-separable spaces of Lipschitz functions}

Note that in Theorem~\ref{hypersurfaces}, the Banach space $V$ 
is assumed to be \emph{separable}, in order to invoke
the theorem of Zaj\'i\v{c}ek \cite{zajicek1}.
Following the suggestion of a referee,
we show here that a version of Theorem~\ref{hypersurfaces}
can be obtained for certain non-separable Banach spaces,
where our notion of hypersurface (i.e.~the graph of a delta-convex function over a hyperplane)
is replaced by the wider notion of
\emph{Lipschitz hypersurface}
(i.e.~the graph of a Lipschitz function over a hyperplane).
Following Zaj\'i\v{c}ek \cite{zajicek1983},
a set that can be covered by countably many Lipschitz hypersurfaces
is called \emph{sparse}\footnote{The terminology is inspired by the use of this term by Blumberg \cite{blumberg}.}.
To fix ideas, we focus attention on the key special case $V = \mathrm{Lip}(X)$,
consisting of real-valued Lipschitz functions on $X$,
equipped with the norm
$\|f\|_{\mathrm{Lip}} = \|f\|_\infty +|f|_{\lip}$
(where $|f|_{\lip}:= \sup\{ \frac{|f(x)-f(y)|}{d(x,y)} : x,y\in X, x\neq y\}$),
a space that is continuously included in $C(X)$, but non-separable in general.

\begin{thm}\label{prop:sparse}
If $\, V = \mathrm{Lip}(X)$
then
$V^{!}= \{f\in V: f\text{ has a unique maximizing measure} \}$
is such that $V \setminus V^!$ is sparse,
i.e.~can be covered by countably many Lipschitz hypersurfaces in $V$.
\end{thm}
\begin{proof}
Let $W$ denote the $1$-Wasserstein distance on $\m$,
defined by
\[
W(\mu_1,\mu_2)
= \sup\!\left\{ \int \phi\,\mathrm{d}(\mu_2 - \mu_1) :
\phi \in \mathrm{Lip}(X),\, |\phi|_{\lip} \leqslant 1 \right\},
\]
which is compatible with the weak$^*$ topology on $\m$.
The compact metric space $(\m,W)$ is separable, so choose a
countable dense subset $\m_0 \subseteq \m$.
If $f \in V \setminus V^!$, there exist two distinct $f$-maximizing measures,
which can be separated by disjoint closed Wasserstein balls
$B_i = \bar{B}_{W}(\mu_i, r_i)=\{\nu\in\m:W(\mu_i,\nu)\le r_i\}$, $i = 1,2$,
with $r_1 + r_2 < W(\mu_1,\mu_2)$.
Since $\m_0$ is dense and the radii $r_i$ can be chosen to be rational,
we may assume that $\mu_i \in \m_0$ and $r_i \in \mathbb{Q}$.
The resulting family of pairs $(B_1^{(k)}, B_2^{(k)})_{k=1}^\infty$
is countable, and $V \setminus V^! \subseteq \bigcup_{k=1}^\infty S_k$,
where we define
$$
\beta_i^{(k)}(f) := \sup_{\mu \in B_i^{(k)}} \int f\,\mathrm{d}\mu \ , \qquad
S_k := S_{B_1^{(k)},B_2^{(k)}} = \{ f \in V : \beta_1^{(k)}(f) = \beta_2^{(k)}(f) \}.
$$
It remains to show that each $S_k$ is a Lipschitz hypersurface.
Fix a pair $(B_1^{(k)}, B_2^{(k)})$, and write $B_i = B_i^{(k)}$,
$\beta_i = \beta_i^{(k)}$, $S = S_k$.
Defining $H := \{ f \in V : \int f\,\mathrm{d}\mu_1 = \int f\,\mathrm{d}\mu_2 \}$,
we will show $S$ is the graph of a Lipschitz function defined on
the closed hyperplane $H$.

Since $r_1 + r_2 < W(\mu_1,\mu_2)$, there exists
a $1$-Lipschitz function $\phi \in V$
such that 
$\delta := \int \phi\,\mathrm{d}(\mu_2-\mu_1) - r_1 - r_2 > 0$.
Note that $\phi\notin H$, since $\int\phi\,\mathrm{d}(\mu_2-\mu_1)>0$,
so $V = H\oplus\mathrm{span}\{\phi\}$ and every $f\in V$ decomposes
uniquely as $f = h + t\phi$, for $h\in H$ and $t\in\R$.
For every $v \in V$, define
$
F_v(t) := \beta_1(v + t\phi) - \beta_2(v + t\phi).
$
Now $t\mapsto \beta_i(v+t\phi)$ is a supremum of affine functions of $t$
with slopes $\int\phi\,\mathrm{d}\mu$, for $\mu\in B_i$,
and
$|\int\phi\,\mathrm{d}\mu - \int\phi\,\mathrm{d}\mu_i| \leqslant r_i$
for all $\mu\in B_i$, so
every one-sided derivative of $t\mapsto\beta_2(v+t\phi)$ exceeds
every one-sided derivative  of $t\mapsto\beta_1(v+t\phi)$ by at least $\delta$.
In other words, every one-sided derivative  of $F_v$ is $\leqslant -\delta < 0$,
so
$F_v(t) - F_v(t') \leqslant -\delta(t-t')$ for all $t > t'$,
and in particular
\begin{equation}\label{delta_decrease}
|F_v(t) - F_v(t')| \geqslant \delta\,|t - t'| \quad
\text{for all }t, t' \in \R.
\end{equation}
Hence $F_v$ is strictly decreasing with $F_v(t) \to \pm\infty$ as $t\to\mp\infty$,
so there exists a unique $r(v)\in\R$ such that $v + r(v)\phi \in S$.
To see that $v\mapsto r(v)$ is Lipschitz, let $v,v'\in V$.
Setting $t = r(v)$ and $t' = r(v')$ in (\ref{delta_decrease}) gives
\begin{equation}\label{2.2}
\delta\,|r(v) - r(v')| \leqslant |F_v(r(v')) - F_v(r(v))| = |F_v(r(v'))|.
\end{equation}
But $|F_v(r(v'))| = |F_v(r(v')) - F_{v'}(r(v'))|$, and
$$
|F_v(r(v')) - F_{v'}(r(v'))|
\leqslant |\beta_1(v+r(v')\phi)-\beta_1(v'+r(v')\phi)|
+|\beta_2(v+r(v')\phi)-\beta_2(v'+r(v')\phi)|,
$$
so
\begin{equation}\label{2.4}
|F_v(r(v'))|
\leqslant |\beta_1(v+r(v')\phi)-\beta_1(v'+r(v')\phi)|
+|\beta_2(v+r(v')\phi)-\beta_2(v'+r(v')\phi)|.
\end{equation}
Each $\beta_i$ is 1-Lipschitz on $C(X)$, and hence on $V$, so (\ref{2.4}) gives
\begin{equation}\label{2.3}
|F_v(r(v'))|
\leqslant 2\|v-v'\|_\infty \leqslant 2\|v-v'\|_V.
\end{equation}
Combining (\ref{2.2}), (\ref{2.3}) gives $|r(v)-r(v')| \leqslant (2/\delta)\|v-v'\|_V$,
so $v\mapsto r(v)$ is Lipschitz on $V$.
Therefore
\[
S = \{ h + r(h)\phi : h \in H \}
\]
is a Lipschitz hypersurface in $V$, and the proof is complete.
\end{proof}


\begin{thebibliography}{OOO}



\bibitem{akin} E. Akin, {\it The general topology of dynamical systems}, Amer. Math. Soc., Providence, 1993.

\bibitem{aliprantisborder} C. D. Aliprantis \& K. C. Border,
{\it Infinite-dimensional analysis.
A hitchhiker's guide}, Second edition,
Springer-Verlag, Berlin, 1999.




\bibitem{BL00} Y. Benyamini \& J. Lindenstrauss, 
{\it Geometric nonlinear functional analysis. Vol. 1},
Amer. Math. Soc. Colloq. Publ., 48
American Mathematical Society, Providence, RI, 2000

\bibitem{blumberg} H. Blumberg, Exceptional sets, 
{\it Fund. Math.}, {\bf 32} (1939), 3--32.

\bibitem{Boc18} J. Bochi,
 Ergodic optimization of Birkhoff averages and Lyapunov exponents, {\it Proceedings of the International Congress of Mathematicians, Rio de Janeiro 2018. Vol.
 III. Invited lectures}, 1825--1846, World Sci. Publ., Hackensack, NJ, 2018.



\bibitem{bochi} J. Bochi, Genericity of periodic maximization: Proof of Contreras' theorem following Huang, Lian, Ma, Xu, and Zhang, 2019, available at \url{https://personal.science.psu.edu/jzd5895/docs/Contreras_dapres_HLMXZ.pdf}.

\bibitem{BZ16}
J. Bochi \& Y. Zhang,
Ergodic optimization of prevalent super-continuous functions.
\textit{Int.\ Math.\ Res.\ Not.}, {\bf 19} (2016), 5988--6017.


\bibitem{Bou00}
T. Bousch,
Le poisson n'a pas d'ar\^{e}tes.
\textit{Ann.\ Inst.\ Henri Poincar\'{e} Probab.\ Stat.}, {\bf 36} (2000), 489--508.


	\bibitem{B01} T. Bousch, La condition de Walters,  {\it Ann. Sci. \'{E}cole Norm. Sup.},
{ \bf 34} (2001), 287--311.

    \bibitem{B08} T. Bousch, Nouvelle preuve d'un th\'eor\`eme de Yuan et Hunt,
{\it Bull. Soc. Math. France}, {\bf 136}
    (2008), 227--242.




\bibitem{christensen1}
J. P. R. Christensen,
On sets of Haar measure zero in abelian Polish groups,
{\it Israel J. Math.}, {\bf 13} (1972), 255--260.

\bibitem{christensen2}
J. P. R. Christensen, Measure theoretic zero-sets in infinite dimensional spaces and applications
to differentiability of Lipschitz mappings,
2-i\`eme Colloq. Analyse Fonctionnel (Bordeaux, 1973),
{\it Publications du D\'epartement de Math\'ematiques de Lyon,} 1973, tome {\bf 10}, fascicule 2,
pp. 29--39.

\bibitem{Co16}
G. Contreras,
Ground states are generically a periodic orbit,
\textit{Invent.\ Math.}, {\bf 205} (2016), 383--412.

\bibitem{CLT01}
G. Contreras, A. O. Lopes \& Ph. Thieullen,
Lyapunov minimizing measures for expanding maps of the circle.
\textit{Ergodic Theory Dynam.\ Systems}, {\bf 21} (2001), 1379--1409.

\bibitem{conzeguivarch}
J.-P. Conze \& Y. Guivarc'h,
Croissance des sommes ergodiques et principe variationnel,
{\it manuscript}, circa 1993.

\bibitem{DLZ24}
        J. Ding, Z. Li \& Y. Zhang,
        On the prevalence of the periodicity of maximizing measures,
        \textit{Adv.\ Math.},  {\bf 438} (2024), 109485.

 \bibitem{elekesnagy}
M. Elekes \& D. Nagy,
Haar null and Haar meager sets: a survey and new results,
{\it Bull. Lond. Math. Soc.}, {\bf 52} (2020), 561--619.

\bibitem{engelking}
R.~Engelking,
{\it General topology},
Second edition,
Sigma Ser. Pure Math., 6,
Heldermann Verlag, Berlin, 1989.

\bibitem{fort2}
M. K. Fort, Jr.,
Points of continuity of semi-continuous functions,
{\it Publ. Math. Debrecen}, {\bf 2} (1951), 100--102.

\bibitem{GSZ25} R. Gao, W. Shen \& R. Zhang,
Typicality of periodic optimization over an expanding circle map,
{\it Preprint}, (2025), (arXiv:2501.10949v1).

\bibitem{gallavottimiraclesole}
G.~Gallavotti \& S.~Miracle-Sole, 
Statistical mechanics of lattice systems,
{\it Comm. Math. Phys.}, {\bf 5} (1967), 317--323.

\bibitem{glasner} E. Glasner, {\it Ergodic theory via joinings},
  Mathematical surveys \& monographs 101, Amer. Math. Soc., 2003.

\bibitem{hjxz} W. Huang, O. Jenkinson, L. Xu \& Y. Zhang, 
Typical periodic optimization for dynamical systems: symbolic dynamics, {\it Invent. Math.}, 
{\bf 245} (2026), 1--63.


\bibitem{HLMXZ}
	W. Huang, Z. Lian, X. Ma, L. Xu \& Y. Zhang, Ergodic optimization theory for a class of typical maps,
	{\it J. Eur. Math. Soc.}, doi:10.4171/JEMS/1652


\bibitem{HO96a}
	B. R. Hunt \& E. Ott, Optimal periodic orbits of chaotic systems,
{\it Phys. Rev. Lett.}, {\bf 76} (1996), 2254--2257.
	
	\bibitem{HO96b}
	B. R. Hunt \& E. Ott, Optimal periodic orbits of chaotic systems occur at low period,
	{\it Phys. Rev. E}, {\bf  54} (1996), 328--337.
		

\bibitem{huntsaueryorke}
B. R. Hunt, T. Sauer \& J. A. Yorke,
Prevalence: a translation-invariant `almost every' on infinite-dimensional spaces,
{\it Bull. Amer. Math. Soc.},
{\bf 27} (1992), 217--238.


\bibitem{huntsaueryorkeaddendum}
B. R. Hunt, T. Sauer \& J. A. Yorke,
 Prevalence. An addendum to: ``Prevalence: a translation-invariant `almost every' on infinite-dimensional spaces'',
{\it Bull. Amer. Math. Soc.}, {\bf 28} (1993), 306--307.

\bibitem{israel} R. B. Israel, {\it Convexity in the theory of lattice gases}, Princeton Ser. Phys., Princeton
University Press, Princeton, NJ, 1979.

\bibitem{israelphelps} R. B. Israel \& R. R. Phelps,
Some convexity questions arising in statistical mechanics,
{\it Math. Scand.}, {\bf 54} (1984), 133--156.

\bibitem{Jen00}
O. Jenkinson,
Frequency locking on the boundary of the barycentre set,
{\it Exp. Math.}, {\bf 9} (2000), 309--317.
	

\bibitem{jenkinsonergodicoptimization}
O. Jenkinson, Ergodic optimization,
{\it Discrete \& Cont. Dyn. Sys.}, {\bf 15} (2006), 197--224.


\bibitem{new survey}
O. Jenkinson,
Ergodic optimization in dynamical systems,
\textit{Ergodic Theory Dynam.\ Systems}, {\bf 39} (2019), 2593–-2618.

\bibitem{jllz} 
O.~Jenkinson, X.~Li, Y.~Liao \& Y.~Zhang,
Ergodic optimization for open expanding multi-valued topological dynamical systems,
{\it Preprint}, (2025), 
arXiv:2503.18092.

\bibitem{kb} N. Krylov \& N. Bogolioubov, La th\'eorie
  g\'en\'erale de la mesure dans son application \`a l'\'etude des
  syst\`emes dynamiques de la m\'ecanique non lin\'eaire, {\it Ann.
    Math.}, {\bf 38} (1937), 65--113.

\bibitem{LZ24} Z.~Li \& Y.~Zhang,
Ground states and periodic orbits for expanding Thurston maps,
{\it Math. Ann.}, 
{\bf 391} (2025), 3913--3985.

\bibitem{mazur} S. Mazur, Uber konvexe Mengen in linearen normierten Raumen,
{\it Studia Math.}, {\bf 4} (1933),
70--84.

\bibitem{millerakin} W. Miller \& E. Akin,
Invariant measures for set-valued dynamical systems, {\it Trans. Amer. Math. Soc.}, {\bf 351} (1999), 1203--1225.


\bibitem{Mo13}
W. B. Moors, 
A note on Fort's theorem,
{\it Topology Appl.},  {\bf 160} (2013), 305--308.


\bibitem{morrisnonlinearity} I. D. Morris,
Maximizing measures of generic H\"older functions have zero entropy,
{\it Nonlinearity},
{\bf 21} (2008), 993--1000.

\bibitem{morris} I. D. Morris, Prevalent uniqueness in ergodic optimisation,
{\it Proc. Amer. Math. Soc.}, {\bf 149} (2021),
1631--1639.


\bibitem{pavlica} D. Pavlica,
On the points of non-differentiability of convex functions,
{\it Comment. Math. Univ. Carolin.}, {\bf 45} (2004), 727--734.



\bibitem{PZ01}
D. Preiss  \& L. Zaj\' i\v cek,
Directional derivatives of Lipschitz functions,
{\it Israel J. Math.},  {\bf 125} (2001), 1--27.

 
\bibitem{quassiefken}
A. Quas \& J. Siefken, 
Ergodic optimization of supercontinuous functions on shift spaces, 
{\it Ergodic Theory Dynam. Systems}, {\bf 32} (2012), 2071--2082.

\bibitem{ruelle} D.~Ruelle, {\it Thermodynamic Formalism},
  Reading, Mass., Addison-Wesley, 1978.


\bibitem{veselyzajicek}
L. Vesely \& L. Zaj\' i\v cek,
On differentiability of convex operators,
{\it J. Math. Anal. \& Appl.},
{\bf 402} (2013), 12--22.



\bibitem{walters} P. Walters, {\it An introduction to ergodic
    theory}, Graduate texts in mathematics 79, Springer, 1981.



\bibitem{zajicek1} L. Zaj\' i\v cek, On the differentiation of convex functions in finite and infinite dimensional spaces,
{\it Czech. Math. J.},
{\bf 29} (1979), 340--348.

\bibitem{zajicek1983} L.\ Zaj\'i\v{c}ek,
Differentiability of the distance function and points of multi-valuedness
of the metric projection in Banach space,
{\it Czechoslovak Math.\ J.}, {\bf 33} (1983), 292--308.


\bibitem{zajicek3} L. Zaj\' i\v cek,
On $\sigma$-porous sets in abstract spaces,
{\it Abstr. Appl. Anal.}, {\bf 5} (2005), 509--534.




\end{thebibliography}
\end{document}